\documentclass[american]{amsart}

\usepackage[T1]{fontenc}
\usepackage[latin1]{inputenc}
\usepackage[english]{babel}

\usepackage{amsfonts}
\usepackage{amssymb}
\usepackage{amsthm}
\usepackage{amsmath}
\usepackage{mathrsfs}
\usepackage[all]{xy}
\usepackage{graphicx,psfrag,mathrsfs}
\usepackage{hyperref}
\usepackage{svgcolor,comment}
\usepackage[colorinlistoftodos]{todonotes} 

\everymath{\displaystyle}

\newtheorem{theorem}{Theorem}
\newtheorem*{main-result}{Theorem}

\newtheorem{corollary}[theorem]{Corollary}

\theoremstyle{definition}
\newtheorem{question}{Question}

\newtheorem{def-theo}[theorem]{Definition-Property}
\theoremstyle{remark}

\newtheorem{claim}{Claim}

\DeclareMathOperator{\id}{id}

\DeclareMathOperator{\dd}{\mathbb{D}}

\DeclareMathOperator{\utei}{\mathbf{T}}
\DeclareMathOperator{\teiq}{\mathbf{T}^q}
\DeclareMathOperator{\stei}{\mathbf{T}_s}

\newcommand{\Mod}[1]{\mathrm{Mod}\!\left(#1\right)}
\newcommand{\modred}[2]{\mathrm{M}^{\textrm{red}}\!\left(#1, #2\right)}
\newcommand{\tei}[1]{\mathbf{T}^{#1}}
\newcommand{\inthyp}[1]{\mathrm{L}^{#1}\!\left( \mathbb{D}, \sigma \right)}

\begin{document}

\title[On a smoothness result]{On Smoothness of the elements of some integrable Teichm\"uller spaces}
\author{Vincent Alberge and Melkana Brakalova}

\begin{abstract}
In this paper we focus  on the integrable Teichm\"uller spaces $\tei{p}$ ($p>0$) which are subspaces of the symmetric subspace of the universal Teichm\"uller space. We prove that any element of   $\tei{p}$ for $0<p\leq 1,$  is a $\mathcal{C}^1$-diffeomorphism.
\end{abstract}


\maketitle

\section{Introduction}

The universal Teichm\"uler space $\utei$ is the space of \emph{quasisymmetric} homeomorphisms of the unit circle $\mathbb{S}^1$ fixing $1$, $i$, and $-1$. A mapping $f: \mathbb{S}^1\rightarrow \mathbb{S}^1$ is said to be quasisymmetric if  there exists $M>0$ such that 
$$
\forall \theta\in\mathbb{R},\forall t>0,\; \frac{1}{M}\leq\left|\frac{f(e^{i(\theta+t)})-f(e^{i\theta})}{f(e^{i\theta})-f(e^{i(\theta-t)})}\right|\leq M .
$$
Due to a well-known  result by Ahlfors and Beurling  \cite{beurling&ahlfors} one can give an equivalent description of $\utei$. More precisely, the universal Teichm\"uller space can be defined as  the set of \emph{Teichm\"uller equivalence classes} of \emph{quasiconformal mappings} of the unit disc $\dd$ fixing $1$, $i$, and $-1$ where two such mappings are Teichm\"uller equivalent if they coincide on $\mathbb{S}^1$.  A mapping $F:D \rightarrow F(D),$ where $D\subset \mathbb C$ is a domain,  is called quasiconformal (or q.c. for short) if it is  an orientation-preserving homeomorphism  and if its distributional derivatives $\partial_z F$ and $\partial_{\overline{z}}F$ can be represented by locally square integrable  functions (also denoted by $\partial_z F$ and $\partial_{\overline{z}}F$) on $D$ such that 
$$
\left\Vert \frac{\partial_{\overline{z}}F}{\partial_z F} \right\Vert_{\infty} = \underset{z\in D}{\textrm{ ess.sup}}\left| \frac{\partial_{\overline{z}}F\left(z\right)}{\partial_z F\left( z \right)} \right| <1.
$$
We also recall that for $z=x+iy$, $\partial_{\overline{z}}=\tfrac{1}{2}(\partial_{x}+i\partial_{y})$ and $\partial_{z}=\tfrac{1}{2}(\partial_{x}-i\partial_{y})$. Furthermore, if $F$ is a quasiconformal mapping, the function $\mu_F =\tfrac{\partial_{\overline{z}}F}{\partial_z F},$ defined a.e., is called the \emph{Beltrami coefficient} associated with $F$. By the measurable Riemann mapping theorem,  if a measurable function $\mu$ on $D$ is such that  $\left\Vert \mu \right\Vert_{\infty}<1$, then it is the Beltrami coefficient of some quasiconformal mapping, which we will  denote  here by $F^{\mu}$. 

Let us now introduce an important subspace of $\utei$, namely, the \emph{symmetric Teichm\"uller space} denoted here by $\stei$. Following a terminology introduced by Gardiner and Sullivan  \cite{gardiner&sullivan}, it is the space of \emph{symmetric} homeomorphism of $\mathbb{S}^1$ fixing $1$, $i$, and $-1$. One recalls that $f:\mathbb{S}^1\rightarrow \mathbb{S}^1$ is symmetric if it is an orientation-preserving homeomorphism of $\mathbb{S}^1$ such that 
\begin{equation}\label{eq:0}
\frac{f(e^{i(\cdot+t)})-f(e^{i\cdot})}{f(e^{i\cdot})-f(e^{i(\cdot-t)})} \underset{t\rightarrow 0^+}{\longrightarrow} 1,
\end{equation}
with respect to the uniform convergence  on $\mathbb{R}$. As for the universal Teichm\"uller space one has an equivalent description of such a space that involves quasiconformal mappings. Indeed, Gardiner and Sullivan proved (see Theorem 2.1 in \cite{gardiner&sullivan}) that $\stei$ corresponds to the space of Teichm\"uller equivalent classes of quasiconformal mappings of $\dd$ fixing $1$, $i$, and $-1$ admitting a representative which is \emph{asymptotically conformal} on $\mathbb{S}^1$. Let us recall that a quasiconformal mapping $F: \dd \rightarrow \dd$ is said to be asymptotically conformal on $\mathbb{S}^1$ if for every $\epsilon>0$, there exists a compact subset $K_{\epsilon}$ of $\dd$ such that for any $z\in \dd\setminus K_{\epsilon}$, $\left| \mu_{F}(z)\right| <\epsilon$.  

Here we focus on some  interesting infinite dimensional  subspaces of $\utei$, the $p$-\emph{integrable Teichm\"uller spaces}, which we define for each $p>0$ as the set 
$$
\tei{p}=\left\lbrace f\in\utei\mid \exists F:\dd\rightarrow \dd, \textrm{ q.c. such that } F_{\vert_{\mathbb{S}^1}} = f \textrm{ and }\mu_F \in\inthyp{p}  \right\rbrace,
$$
where $\sigma$ is the hyperbolic  measure on $\dd$, that is, for any $z=x+iy \in\dd$, $d\sigma (z) = (1-\left| z \right|^2)^{-2}dxdy$.  It is elementary to observe from such a definition that if $q>p>0$, then $\tei{p} \subset \teiq$. The  spaces $\tei{p}$, $p\geq 2$, were  first introduced by Guo  \cite{guo} through an equivalent description involving univalent functions. At about the same time, Cui  \cite{cui} studied the case $p=2$ and gave a few important characterizations of the elements of $\tei{2}$. In particular, he proved that the Beltrami coefficient associated with the \emph{Douady--Earle extension} (see \cite{douady&earle}) of any element of $\tei{2}$ belongs to $\inthyp{2}$. Later on, Takhtajan and Teo  \cite{takhtajan&teo} introduced a  Hilbert manifold structure on the universal Teichm\"uller space that makes the space $\tei{2}$ the connected component of the identity mapping $\id_{\mathbb{S}^1}$. With respect to such a structure, they proved that the so-called  \emph{Weil--Petersson} metric is a Riemannian metric on $\utei$. Following Takhtajan and Teo's work, the space $\tei{2}$ is now reffered to as the \emph{Weil--Petersson Teichm\"uller space}. For further results on $\tei{2}$ we refer to \cite{shen}.   Let us point out that one can obtain $\tei{2}\subset\stei$  by combining  \cite[Theorem 2 and Lemma 2]{cui} and \cite[Theorem 4]{earle&markovic&saric}, see \cite[Section 3]{fan&hu} for a more detailed explanation. One can also mention the paper \cite{tang} by Tang where in particular, Cui's result concerning  the Douady--Earle extension is extended to all spaces $\tei{p}$ with $p\geq 2$.  Recently, the second author of this paper proved in  \cite{brakalovaxx}    that $\tei{2}\subset\stei$ using an approach based on module techniques and the so-called  \emph{Teichm\"uller's Modulsatz} (see \cite[\S 4]{T13}), and later on using a different method  she proved that for any $p>0$, $\tei{p}\subset \stei$ (see \cite{brakalova18}). 

In this paper we only deal with  $\tei{p}$ for $0<p\leq 1$ and we give  a proof of the following result:
\begin{theorem}\label{main-result}
Let $p\leq1$. Then, any element of $\tei{p}$ is a $\mathcal{C}^1$-diffeomorphism.
\end{theorem}

The strategy of the proof takes advantage of an approach used by the second author of this paper and J. A. Jenkins \cite{brakalova&jenkins02}, modified  to the case of the unit disc.  We  first use the  \emph{Teichm\"uller--Wittich--Bellinski\u{\i}} to show that each  element of $\tei{1}$  has a non-vanishing derivative at each point of $\mathbb{S}^1$. Then, we use properties of the \emph{reduced module} of a simply-connected domain  to show that the derivatives of  the elements of $\tei{1}$  are continuous. As mentioned earlier, since $\tei{p} \subset \tei1$ for $0<p\leq 1$, it follows immediately that for $0<p\leq 1$, any element in $\tei{p}$ is continuously differentiable with non-vanishing derivative.

\section{Background}

In  this section we  recall some classic  notions from  geometric function theory. Such notions are most notably and thoroughly investigated in Teichm\"uller's \emph{Habilitationsschrift} (Habilitation Thesis) \cite{T13}. 

\subsection{Module of a doubly-connected domain} Let $D$ be a (non-degenerate) doubly-connected domain of the extended complex plane, that is, the complement of $D$ is an union of two disjoint  simply-connected domains, each bounded by a Jordan curve. It is well known (see \cite{lehto&virtanen,T13}) that there exists a biholomorphic function that maps $D$ onto an annulus of inner radius $r_1$ and outer radius $r_2$ for some $0<r_2<r_1<\infty$. The \emph{module} $\Mod{D}$ of $D$  is $\ln\left( \tfrac{r_2}{r_2}\right)$. It is a \emph{conformal invariant}, namely, if $\Psi : D \rightarrow \Psi(D)$ is a biholomorphic function, then $\Mod{D}=\Mod{\Psi(D)}$. 

It is also well known (see \cite{lehto&virtanen,T13}) that the module is \emph{superadditive}. More precisely, if $D_1$ and $D_2$ are two disjoint doubly-connected subdomains of a doubly-connected domain $D_3$, where each separates some $z_0 \in\mathbb{C}$ from $\infty$, then 
\begin{equation}\label{eq:21}
\Mod{D_1}+\Mod{D_2}\leq \Mod{D_3}.
\end{equation}
In saying that a doubly-connected domain separates $z_0$ from $\infty$, we mean that one component of its complement  contains $z_0$ in its interior while the  other component  contains $\infty$.

Let us now recall two inequalities that will be used in the proof of the main result. For $0<r_2<r_1$  and $\zeta \in\mathbb{C}$ we set $A_{\zeta,r_2 , r_1}=\left\lbrace z \mid r_2 < \left| z -\zeta \right| < r_1\right\rbrace$. Let $F:A_{\zeta,r_2 , r_1} \rightarrow F\left(A_{\zeta,r_2 , r_1} \right)$ be a quasiconformal mapping. Then setting $z=\zeta+re^{i\theta}, r_2<r<r_1$ we have
\begin{equation}\label{eq:22}
\Mod{F\left(A_{\zeta , r_2 , r_1}\right)}\leq \frac{1}{2\pi}\iint_{A_{\zeta , r_2 , r_1}}{\frac{1+\left| \mu_{F}(z)\right|}{1-\left| \mu_{F}(z)\right|}\cdot\frac{dxdy}{\left| z-\zeta\right|^2}},
\end{equation}
and 
\begin{equation}\label{eq:23}
2\pi\int_{r_2}^{r_1}{\frac{1}{\int_{0}^{2\pi}{\frac{1+\left| \mu_{F}\left(z \right)\right|}{1-\left| \mu_{F}\left(z \right)\right|}d\theta}}\cdot\frac{dr}{r}}\leq \Mod{F\left(A_{\zeta , r_2 , r_1}\right)}.
\end{equation}
These estimates could be obtained  following Teichm\"uller's approach based on  the  \emph{length-area method}  in \cite[\S 6.3]{T13},  where he arrived at  weaker versions of (\ref{eq:22}) and (\ref{eq:23}).  Estimates equivalent to (\ref{eq:22}) and (\ref{eq:23})---some proved under more general assumptions and different methods---can be found in  \cite{reich&walczak,gutlyanskii&martio,brakalova10} and others.

\subsection{Reduced module of a simply-connected domain} Let $\Omega$ be a simply-connected domain of the complex plane different from $\mathbb{C}$. Let $\zeta \in \Omega$. For $r>0$,  let  $D(\zeta,r)$  denote the disc of radius $r$ centered at $\zeta$ and let $0<r_2<r_1$ be  small enough so that $D(\zeta,r_1)\subset \Omega$. From (\ref{eq:21})  follows$$
\Mod{\Omega\setminus D(\zeta,r_1)} + \ln\left( \frac{r_1}{r_2}\right) \leq \Mod{\Omega\setminus D(\zeta,r_2)},
$$
and therefore
$$
\Mod{\Omega\setminus D(\zeta,r_1)} +\ln\left( r_1\right) \leq \Mod{\Omega\setminus D(\zeta,r_2)} +\ln\left( r_2\right).
$$
One defines the reduced module $\modred{\Omega}{\zeta}$ of $\Omega$ at $\zeta$ as  $\lim_{r\rightarrow 0} \Mod{\Omega \setminus D(\zeta,r)}+\ln(r)$. Using, for example, \emph{Koebe distortion theorem} one can show that this limit is finite and $\modred{\Omega}{\zeta} = \ln\left( \left| \Psi^{\prime}(0)\right|\right)$, where $\Psi :  \dd \rightarrow \Omega$ is a biholomorphic function mapping $0$ onto $\zeta$. A detailed proof can be found in  \cite[§1.6]{T13}. From here it follows directly that  $\zeta\mapsto \modred{\Omega}{\zeta}$ is continuous.

Before concluding this subsection let us add one more  property of the reduced module that we will use later. 

If $F:\mathbb{C}\rightarrow \mathbb{C}$ is a homeomorphism then, for any $r>0,$ the function $\zeta\mapsto \modred{F\left( D(\zeta,r) \right)}{F(\zeta)}$ is continuous. Indeed, if $\zeta_n \underset{n\rightarrow \infty}{\longrightarrow}\zeta$, then  by applying a sequence of  biholomorphic functions $z\mapsto F(z+\zeta_n-\zeta)-F(\zeta_n)+F(\zeta), z\in D(\zeta,r),$ one obtains a sequence of domains $D_n,$ which  are all images of $D(\zeta,r)$. Since $F(z)$ is a homeomorphisms  it follows that  $D_n \underset{n\rightarrow \infty}{\longrightarrow} F\left( D(\zeta,r)\right)$ (with respect to the topology induced by the Hausdorff distance on the set of subsets of $\mathbb{C}$). Consider the sequence of biholomorphic functions $\Psi_n : \mathbb{D} \rightarrow D_n$ mapping $0$ onto $F(\zeta),$ normalized by $\Psi_n'(0)>0$. Then for any $n$, $\ln\left(  \Psi_{n}^{\prime}(0)\right)=\modred{D_n}{F(\zeta)}=\modred{F\left( D(\zeta_n,r) \right)}{F(\zeta_n)}$ since a translation does not change the reduced module.  Furthermore,  the sequence of functions $\Psi_n$ forms a normal family and thus, up to a subsequence, $\Psi_n$ converges uniformly (on any compact subset of $\mathbb{D}$) to a biholomorphic function $\Psi_{\infty}: \mathbb{D}\rightarrow F\left( D(\zeta,r) \right)$ mapping $0$ onto $\zeta$. This implies   
\begin{align*}
\modred{F\left(D(\zeta,r)\right)}{F(\zeta)} & =\ln\left(  \Psi_{\infty}^{\prime}(0)\right) \\ & = \lim_{n\rightarrow \infty}\ln\left(  \Psi_{n}^{\prime}(0) \right) \\ & =\lim_{n\rightarrow \infty}\modred{F\left( D(\zeta_n,r) \right)}{F(\zeta_n)},
\end{align*}
and thus we have continuity.

\subsection{Teichm\"uller--Wittich--Bellinski\u{\i} theorem} First, let us recall that a mapping $F:\mathbb{C}\rightarrow\mathbb{C}$ is said to be \emph{conformal} at $z_0 $ if $\lim_{z\rightarrow z_0}{\tfrac{F(z)-F(z_0)}{z-z_0}}$ exists and is different from $0$. Following \cite[Chapter V, Theorem 6.1]{lehto&virtanen} the well-known  Teichm\"uller--Wittich--Bellinski\u{\i} theorem can be stated as follows:
\begin{theorem}\label{theorem:1}
Let $D$ be a domain of the complex plane and let $z_0\in D$. Let $F:D\rightarrow F(D)$ be a quasiconformal mapping. If there exists a neighborhood $\mathcal{U}$ of $z_0$ contained in $D$ such that 
$$
\iint_{\mathcal{U}}{\frac{\left|\mu_{F}(z)\right|}{\left| z-z_0 \right|^2}dxdy}<\infty;
$$
then $F$ is conformal at $z=z_0$.
\end{theorem}
The history of this theorem and its extensions is rather long and we may refer the curious reader to some of the following  papers \cite{belinskii,reich&walczak,drasin,brakalova&jenkins94,gutlyanskii&martio,brakalova10,shishikura} and to \cite{alberge&brakalova&papadopoulos}.

\section{Proof of the main result}\label{sec:3}

Let  $f\in \tei{1}$. By definition, there exists a quasiconformal extension $F$ of $f$ to the closed unit disc  such that 
\begin{equation}\label{eq:3}
\iint_{\dd}{\left|\mu_{F}(z)\right| d\sigma(z)}<\infty.
\end{equation}

Let $\widetilde{\mu}$ be a function defined on the extended complex plane which coincides with $\mu$ on $\dd$ and which is identically  $0$ outside the disc. Let $F^{\widetilde{\mu}}$ be the unique quasiconformal mapping of the complex plane  with Beltrami coefficient $\widetilde{\mu}$   that fixes  $1$, $i$, and $-i$. Therefore, we have  $F^{\widetilde{\mu}}_{\vert_{\dd}}=F$ and  $F^{\widetilde{\mu}}_{\vert_{\mathbb{S}^1}}=f$. 

\begin{claim}\label{claim1}
The quasiconformal mapping $F^{\widetilde{\mu}}$ is conformal at any point of $\mathbb{S}^1$. Therefore, $f$ is a diffeomorphism of $\mathbb{S}^1$.
\end{claim}

We  apply Theorem \ref{theorem:1} to derive the conformality of $F^{\widetilde{\mu}}$. 

\begin{proof}[Proof of Claim \ref{claim1}]
Let $\zeta_0\in\mathbb{S}^1$. Because of (\ref{eq:3}), one can find a compact subset $K$ of $\dd$ such that 
\begin{equation}\label{eq:4}
\iint_{\dd\setminus K}{\left|\mu_{F}(z)\right| d\sigma(z)}<1.
\end{equation}
Let $r>0$ be such that $\dd \setminus D(\zeta_0, r)\subset \dd\setminus K$.  One first observes that 
\begin{align*}
\forall z \in D(\zeta_0, r)\cap \dd, \; \left(1-\left| z\right|^2 \right)^2& = \left( 1-\left|z\right| \right)^2 \cdot \left( 1 +\left| z\right|\right)^2  \\ & \leq \left| \zeta_0 -z \right|^2 \cdot \left(1+\left| z \right|\right)^2 \\
&< 4\cdot  \left| \zeta_0 -z \right|^2,
\end{align*}
and therefore
\begin{equation}\label{eq:5}
\forall z \in D(\zeta_0, r)\cap \dd, \, \frac{1}{\left| z-\zeta_0\right|^2}< 4\cdot \frac{1}{\left( 1-\left|z\right|^2 \right)^2}.
\end{equation}
It follows 
\begin{align*}
\iint_{D(\zeta_0,r)}{\frac{\left| \widetilde{\mu} (z)\right|}{\left| z-\zeta_0 \right|^2}dxdy}&  = \iint_{D(\zeta_0,r)\cap{\dd}}{\frac{\left| \mu_{F} (z)\right|}{\left| z-\zeta_0 \right|^2}dxdy} \\ & \leq 4 \iint_{D(\zeta_0,r)\cap{\dd}}{\left| \mu_{F} (z)\right| d\sigma(z)}\\ 
& \leq 4.
\end{align*}

We deduce, by Theorem \ref{theorem:1}, that $F^{\widetilde{\mu}}$ is conformal at $z=\zeta_0$ which proves that $f$ is differentiable at $\zeta_0$ and $\left| f^{\prime}(\zeta_0)\right|>0$. Since this is true for any $\zeta_0 \in \mathbb{S}^1,$ we deduce that $f$ is a diffeomorphism of $\mathbb{S}^1$.
\end{proof}

The following two additional results will be needed in the proof of the continuity of  $f^{\prime}$  on $\mathbb{S}^1$. 

\begin{claim}\label{claim2}
Let $\epsilon >0$. Then, there exists $r_{\epsilon}>0$ such that 
$$
\forall \zeta\in\mathbb{S}^1, \forall 0<\rho_2 <\rho_1\leq r_{\epsilon},\;  \left| \Mod{F^{\widetilde{\mu}}\left( A_{\zeta,\rho_2 , \rho_1}\right)} - \ln\left(\frac{\rho_{1}}{\rho_2} \right)\right| < \epsilon.
$$
\end{claim}

\begin{claim}\label{claim3}
Let $\zeta\in\mathbb{S}^1$ and $r>0$. Then,
$$
\lim_{\rho\rightarrow 0}{\Mod{F^{\widetilde{\mu}}\left( A_{\zeta,\rho,r}\right)}+\ln\left( \left| f^{\prime}(\zeta)\right|\rho\right)} = \modred{F^{\widetilde{\mu}}\left( D\left( \zeta, r\right)\right)}{f(\zeta)}.
$$
\end{claim}

\begin{proof}[Proof of Claim \ref{claim2}]
Let $\zeta\in\mathbb{S}^1$ and $0<\rho_2 < \rho_1$.  One the one hand, by applying (\ref{eq:22}) one gets
\begin{align}
\Mod{F^{\widetilde{\mu}}\left( A_{\zeta,\rho_2 , \rho_1}\right)} - \ln\left(\frac{\rho_{1}}{\rho_2} \right) & \leq \frac{1}{2\pi}\iint_{A_{\zeta , \rho_2 , \rho_1}}{\frac{1+\left| \widetilde{\mu}(z)\right|}{1-\left| \widetilde{\mu}(z)\right|}\cdot\frac{dxdy}{\left| z-\zeta\right|^2}} - \ln\left(\frac{\rho_{1}}{\rho_2} \right) \nonumber\\
& = \frac{1}{2\pi}\iint_{A_{\zeta , \rho_2 , \rho_1}}{\left(\frac{1+\left| \widetilde{\mu}(z)\right|}{1-\left| \widetilde{\mu}(z)\right|}-1\right)\cdot\frac{dxdy}{\left| z-\zeta\right|^2}} \nonumber\\
& \leq \frac{1}{\pi \left( 1-\left\Vert \mu_F\right\Vert_{\infty} \right)}\iint_{A_{\zeta , \rho_2 , \rho_1}\cap\dd}{\left| \mu_{F}(z)\right|\cdot \frac{dxdy}{\left|z-\zeta \right|^2}}.\label{eq:31}
\end{align}
On the other hand since 
$$
\int_{0}^{2\pi}{\frac{1+\left| \widetilde{\mu}\left(z \right)\right|}{1-\left| \widetilde{\mu}\left(z \right)\right|}d\theta}\geq 2\pi ,
$$
 by means of (\ref{eq:23}) one obtains
\begin{align}
\Mod{F^{\widetilde{\mu}}\left( A_{\zeta,\rho_2 , \rho_1}\right)} - \ln\left(\frac{\rho_{1}}{\rho_2} \right) & \geq 2\pi\int_{\rho_2}^{\rho_1}{\frac{1}{\int_{0}^{2\pi}{\frac{1+\left| \widetilde{\mu}\left(z \right)\right|}{1-\left| \widetilde{\mu}\left(z \right)\right|}d\theta}}\cdot\frac{dr}{r}} -\ln\left(\frac{\rho_{1}}{\rho_2} \right) \nonumber \\ & = \int_{\rho_2}^{\rho_1}{\frac{2\pi-\int_{0}^{2\pi}{\frac{1+\left| \widetilde{\mu}\left(z \right)\right|}{1-\left| \widetilde{\mu}\left(z \right)\right|}d\theta}}{\int_{0}^{2\pi}{\frac{1+\left| \widetilde{\mu}\left(z \right)\right|}{1-\left| \widetilde{\mu}\left(z \right)\right|}d\theta}}\cdot\frac{dr}{r}} \nonumber \\
& = \int_{\rho_2}^{\rho_1}{\frac{\int_{0}^{2\pi}{\frac{-2\left| \widetilde{\mu}\left(z \right)\right|}{1-\left| \widetilde{\mu}\left(z \right)\right|}d\theta}}{\int_{0}^{2\pi}{\frac{1+\left| \widetilde{\mu}\left(z \right)\right|}{1-\left| \widetilde{\mu}\left(z \right)\right|}d\theta}}\cdot\frac{dr}{r}} \nonumber\\
& \geq \frac{-1}{\pi}  \iint_{A_{\zeta , \rho_2 , \rho_1}\cap\dd}{\frac{\left| \mu_{F}(z)\right|}{1-\left| \mu_{F}\left( z \right)\right|}\cdot \frac{dxdy}{\left|z-\zeta \right|^2}}\nonumber \\
& \geq -\frac{1}{\pi \left( 1-\left\Vert \mu_F\right\Vert_{\infty} \right)}\iint_{A_{\zeta , \rho_2 , \rho_1}\cap\dd}{\left| \mu_{F}(z)\right|\cdot \frac{dxdy}{\left|z-\zeta \right|^2}}.\label{eq:32}
\end{align}
Let $\epsilon>0$. Still because of (\ref{eq:3}) there exists a compact set $K_{\epsilon}$ of $\dd$ such that 
\begin{equation}\label{eq:6}
\iint_{\dd\setminus K_{\epsilon}}{\left| \mu_{F}(z)\right| d\sigma(z)} <\frac{\pi(1-\left\Vert \mu_{F} \right\Vert_{\infty})}{4}\epsilon.
\end{equation}
Let $r_{\epsilon}>0$ be the distance between $\mathbb{S}^1$ and $K_{\epsilon}$. Thus, for any $0<\rho_2 < \rho_1\leq r_{\epsilon}$ one obtains by combining (\ref{eq:31}), (\ref{eq:32}), (\ref{eq:5}) and (\ref{eq:6})  
$$
\forall \zeta\in\mathbb{S}^1,\; -\epsilon < \Mod{F^{\widetilde{\mu}}\left( A_{\zeta,\rho_2 , \rho_1}\right)} - \ln\left(\frac{\rho_{1}}{\rho_2} \right) < \epsilon,
$$
and therefore  Claim \ref{claim2} follows.
\end{proof}

\begin{proof}[Proof of Claim \ref{claim3}]
Let $\zeta\in\mathbb{S}^1$ and let $r>0$. For any $0<\rho<r$, let 
$$m(\rho)=\min_{\left| z -\zeta\right|=\rho}{\left| F^{\widetilde{\mu}}(z)-f(\zeta)\right|} \textrm{ and } M(\rho)=\max_{\left| z -\zeta \right|=\rho}{\left| F^{\widetilde{\mu}}(z)-f(\zeta)\right|}.
$$
Since $F^{\widetilde{\mu}}$ is conformal  at $\zeta$ one has
\begin{equation}\label{eq7}
\lim_{\rho\rightarrow 0}\frac{\left| f^{\prime}(\zeta)\right|\rho}{M(\rho)}=\lim_{\rho\rightarrow 0}\frac{\left| f^{\prime}(\zeta)\right|\rho}{m(\rho)}= 1.
\end{equation}
Furthermore, it is evident that 
\begin{align*}
\Mod{F^{\widetilde{\mu}}\left( D(f(\zeta),r )\right) \setminus D(\zeta, M(\rho))} & \leq \Mod{F^{\widetilde{\mu}}\left( A_{\zeta, \rho, r} \right)} \\ & \hphantom{dsds} \leq \Mod{F^{\widetilde{\mu}}\left( D(\zeta,r )\right)\setminus D(f(\zeta), m(\rho))}.
\end{align*}
Therefore, by adding $\ln\left(\left| f^{\prime}(\zeta)\right|\rho\right)$, using (\ref{eq7}), and letting $\rho \rightarrow 0$ it follows that 
$$
\lim_{\rho\rightarrow 0} \Mod{F^{\widetilde{\mu}}\left( A_{\zeta, \rho, r} \right)}+\ln\left(\left| f^{\prime}(\zeta)\right|\rho\right) =  \modred{F^{\widetilde{\mu}}\left( D\left( \zeta, r\right)\right)}{f(\zeta)},
$$
which proves Claim \ref{claim3}.
\end{proof}

We have now all the ingredients necessary to complete the proof of our main  Theorem \ref{main-result}. 

Let $\zeta_0 \in\mathbb{S}^1$.  Let $\epsilon>0$. Let $r_{\frac{\epsilon}{5}}>0$ be as in Claim \ref{claim2}. By the continuity of the reduced module discussed earlier one can find a $\delta_{\frac{\epsilon}{5}}>0$ such that if $\zeta\in\mathbb{S}^1$ and $\left|\zeta -\zeta_0\right|<\delta_{\frac{\epsilon}{5}}$ then 
\begin{equation}\label{eq8}
\left| \modred{F^{\widetilde{\mu}}\left( D(\zeta, r_{\frac{\epsilon}{5}})\right)}{f(\zeta)}-\modred{F^{\widetilde{\mu}}\left( D(\zeta_0, r_{\frac{\epsilon}{5}})\right)}{f(\zeta_0)}\right| <\frac{\epsilon}{5}.
\end{equation}

Let  $\zeta \in \mathbb{S}^1$ be  such that $\left|\zeta -\zeta_0\right|<\delta_{\frac{\epsilon}{5}}$. By Claim \ref{claim3} there exist $r_{\zeta_0,1}, r_{\zeta,1}<r_{\frac{\epsilon}{5}}$ such that for any $\rho\leq r_{\zeta_0,1}$
\begin{equation}\label{eq9}
\left|  \Mod{F^{\widetilde{\mu}}\left( A_{\zeta_0, \rho, r_{\frac{\epsilon}{5}}} \right)}+\ln\left(\left| f^{\prime}(\zeta_0)\right|\rho\right) - \modred{F^{\widetilde{\mu}}\left( D\left( \zeta_0, r_{\frac{\epsilon}{5}}\right)\right)}{f(\zeta_0)} \right|<\frac{\epsilon}{5} ,
\end{equation}
and for any $\rho\leq r_{\zeta,1}$
\begin{equation}\label{eq10}
\left|  \Mod{F^{\widetilde{\mu}}\left( A_{\zeta, \rho, r_{\frac{\epsilon}{5}}} \right)}+\ln\left(\left| f^{\prime}(\zeta)\right|\rho\right) - \modred{F^{\widetilde{\mu}}\left( D\left( \zeta, r_{\frac{\epsilon}{5}}\right)\right)}{f(\zeta)} \right|<\frac{\epsilon}{5} .
\end{equation}

Thus, from  the triangle inequality, Claim \ref{claim2}, and Inequalities (\ref{eq8}), (\ref{eq9}), and (\ref{eq10}) we obtain
\begin{align*}
& \left| \ln\left( \left| f^{\prime}(\zeta) \right|\right)- \ln\left( \left| f^{\prime}(\zeta_0) \right|\right) \right| \\  &= \left| \ln\left( \left| f^{\prime}(\zeta) \right| r_{\zeta,1}\right)-\ln\left(r_{\zeta,1} \right)- \ln\left( \left| f^{\prime}(\zeta_0) \right|r_{\zeta_0,1}\right) +\ln\left( r_{\zeta_0,1}\right)\right| \\
& \leq \left| \ln\left( \left| f^{\prime}(\zeta) \right| r_{\zeta,1}\right) +\Mod{F^{\widetilde{\mu}}\left( A_{\zeta, r_{\zeta,1}, r_{\frac{\epsilon}{5}}} \right)}- \modred{F^{\widetilde{\mu}}\left( D\left( \zeta, r_{\frac{\epsilon}{5}}\right)\right)}{f(\zeta)}\right|  \\
& \hphantom{g}+\left| \ln\left( \left| f^{\prime}(\zeta_0) \right| r_{\zeta_0,1}\right) +\Mod{F^{\widetilde{\mu}}\left( A_{\zeta_0, r_{\zeta_0,1}, r_{\frac{\epsilon}{5}}} \right)}- \modred{F^{\widetilde{\mu}}\left( D\left( \zeta_0, r_{\frac{\epsilon}{5}}\right)\right)}{f(\zeta_0)}\right| \\
& \hphantom{ghg}+\left| \modred{F^{\widetilde{\mu}}\left( D\left( \zeta, r_{\frac{\epsilon}{5}}\right)\right)}{f(\zeta)}-\modred{F^{\widetilde{\mu}}\left( D\left( \zeta_0, r_{\frac{\epsilon}{5}}\right)\right)}{f(\zeta_0)}\right| \\
& \hphantom{gh}+\left| -\Mod{F^{\widetilde{\mu}}\left( A_{\zeta, r_{\zeta,1}, r_{\frac{\epsilon}{5}}} \right)}-\ln\left( r_{\zeta,1}\right)+ \Mod{F^{\widetilde{\mu}}\left( A_{\zeta_0, r_{\zeta_0,1}, r_{\frac{\epsilon}{5}}} \right)}+\ln\left( r_{\zeta_0,1}\right)\right| \\
& \leq  3\frac{\epsilon}{5} +\left| -\Mod{F^{\widetilde{\mu}}\left( A_{\zeta, r_{\zeta,1}, r_{\frac{\epsilon}{5}}} \right)}+\ln\left( \frac{r_{\frac{\epsilon}{5}}}{r_{\zeta,1}}\right) \right| \\
& \hphantom{dsdsdsdd}+ \left|\Mod{F^{\widetilde{\mu}}\left( A_{\zeta_0, r_{\zeta_0,1}, r_{\frac{\epsilon}{5}}} \right)}-\ln\left( \frac{r_{\frac{\epsilon}{5}}}{r_{\zeta_0,1}}\right)\right| \\
& \leq \epsilon.
\end{align*}

This shows the continuity of $\left| f^{\prime} \right|$ at  any $\zeta_0\in \mathbb{S}^1,$  thus $f^{\prime}$ is continuously differentiable on $\mathbb{S}^1$ and since the derivative is never $0$, any element $f\in \tei{1}$ is a $\mathcal{C}^1$-diffeomorphism on $\mathbb{S}^1$. Since $ \tei{p}\subset \tei{1}$ ($p\leq 1)$ we have shown that Theorem \ref{main-result} holds.

Since every differentiable quasisymmetric function $f$ on $\mathbb{S}^1 $ is  symmetric in the sense of (\ref{eq:0}),  the following already known property follows from Theorem \ref{main-result}.
\begin{corollary}
Let $0<p\leq 1$. Then, $\tei{p}\subset \stei$.
\end{corollary}

Let us point out  that although  $\tei{1}\subset \stei$, the quasiconformal extension $F$ of $f$ we were working with may not necessarily be asymptotically conformal on $\mathbb{S}^1$ and Claim \ref{claim2} is not obvious. However, for  $p\geq 2$,  if one specifically employs  the Douady--Earle extension, then Claim \ref{claim2} holds. It seems natural to ask:
\begin{question}
Let $f\in\tei{p}$ (with $0<p\leq 2$). Is there a quasiconformal asymptotically conformal extension $F$ of $f$ to the closed unit disc  for which $\mu_{F} \in \inthyp{p}$?
\end{question}

Furthermore, since we obtain smoothness  properties for the elements of $\tei{p}$ (for $p\leq 1$), we suggest that one can show higher and higher order of smoothness  for  $p<1,$ as $p$ gets smaller and smaller. If this  is the case we would like to find sharp results on how the order of smoothness depends on $p$, a question that seems to be similar  to finding a characterization of $\tei{p}$ using Sobolev spaces for $p\geq 2$.  In addition, we pose the following question:
\begin{question}
What is $\bigcap_{p>0}\tei{p}$?
\end{question}

\textsc{Vincent Alberge, Fordham University, Department of Mathematics, 441 East Fordham Road, Bronx, NY 10458, USA}

\textit{E-mail address:} { \href{mailto:valberge@fordham.edu}{\tt valberge@fordham.edu} } 

\textsc{Melkana Brakalova, Fordham University, Department of Mathematics, 441 East Fordham Road, Bronx, NY 10458, USA}

\textit{E-mail address:} { \href{mailto:brakalova@fordham.edu}{\tt brakalova@fordham.edu} } 

\end{document}